\documentclass[11pt,letterpaper]{article}
\usepackage[margin=1.5in]{geometry}
\usepackage{amsmath,times}
\usepackage{amssymb,amsmath}
\usepackage{graphicx}
\usepackage{amsthm}
\usepackage{multicol}
\usepackage[all]{xy}
\usepackage{cancel}
\usepackage{tikz}
\usepackage[utf8]{inputenc}
\usepackage{url}

\title{Green relations over finite monoids of $G$-equivariant functions}
\author{Ram\'on H. Ruiz-Medina\footnote{Email: harath.ruiz@academicos.udg.mx}, \  Victor M. Lara-Gómez \\[1em]
\small{Centro Universitario de Ciencias Exactas e Ingenier\'ias}, \\ 
\small{Universidad de Guadalajara, Guadalajara, M\'exico.}}
\date{}

\newtheorem{teorema}{Theorem}[]
\newtheorem{lema}[teorema]{Lemma}
\newtheorem{observacion}[teorema]{Remark}
\newtheorem{corolario}[teorema]{Corollary}
\newtheorem{proposicion}[teorema]{Proposition}
\newtheorem{ejemplo}[teorema]{Example}
\newtheorem{definicion}[teorema]{Definition}

\newtheorem{afirmacion}[teorema]{Claim}

\newcommand{\EndG}{\mathrm{End}_{G}(X)}
\newcommand{\AutG}{\mathrm{Aut}_{G}(X)}
\newcommand{\Hbox}{\mathcal{B}_{[H]}}

\newcommand{\StabG}{\mathrm{Stab}_{G}(X)}
\newcommand{\ConjG}{\mathrm{Conj}(G)}
\newcommand{\ConjX}{\mathrm{Conj}_{G}(X)}

\begin{document}

\maketitle

\begin{abstract}
For a group $G$ acting over a set $X$, the set of all the $G$-equivariant functions, i.e., the set of functions which conmute with the action, ($g\cdot f(x)=g\cdot f(x), \forall g\in G, \forall x\in X$), is a monoid with the composition. The Green Relations are powerful tools to comprehend the structure of a semigroup. We study the case where $X$ is a finite set and compute the green relations for its monoid of $G$-equivariant functions, attempting to describe them based on some particular elements in the monoid called elementary collapsings.  \\

\textbf{Keywords:} Green relations for semigroups, $G$-equivariant function, elementary collapsings. \\

\textbf{MSC 2020:} 20B25, 20E22, 20M20.
\end{abstract}

\section{Introduction}  

For a given group \( G \), a \( G \)-set refers to a set \( X \) upon which \( G \) acts. In other words, there exists a function \( \cdot : G \times X \to X \) such that, for all \( x \in X \), the identity element \( e \) of \( G \) acts as \( e \cdot x = x \), and for all \( g, h \in G \) and \( x \in X \), the action satisfies the condition \( g \cdot (h \cdot x) = (gh) \cdot x \). In semigroup theory, such sets are often referred to as \( G \)-acts. A \( G \)-equivariant transformation, or \( G \)-endomorphism, of a \( G \)-set \( X \) is a function \( \tau : X \to X \) that satisfies \( \tau(g \cdot x) = g \cdot \tau(x) \) for all \( g \in G \) and \( x \in X \). These transformations play a fundamental role in the category of \( G \)-sets and have applications in areas such as equivariant topology, representation theory, and statistical inference.\\

The monoid \( \EndG \) consists of all \( G \)-equivariant transformations of \( X \), with composition as the operation, whereas the group of units \( \AutG \) consists of all bijective \( G \)-equivariant transformations. These objects have been extensively studied in various mathematical contexts (see \cite{cita21}, \cite{cita22}, \cite{cita23}, \cite{cita24}). Numerous cases involving actions of groups and \( G \)-equivariant functions have been examined, particularly in the case of free \( G \)-sets, where the stabilizer of each point is trivial. These sets serve as examples of independence algebras (see \cite{cita11}, \cite{cita12}, \cite{cita13}, \cite{cita14}, \cite{rotman} for further information).\\

Let the set of conjugacy classes of subgroups of \( G \) be denoted as \( \ConjG \), where each conjugacy class is represented as \( [H] := \{ g^{-1} H g : g \in G \} \). When there is a countable number of conjugacy classes, we label them as \( [H_1], [H_2], \dots, [H_r], \dots \), ordered by the size of the subgroups, such that \( |H_1| \leq |H_2| \leq \dots \leq |H_r| \leq \dots \). We also define a finite set with \( r \) elements as \( [r] = \{1, 2, \dots, r\} \). A partial order can be established over \( \ConjG \), given by:
\[
H \leq K \iff \exists g \in G \text{ such that } H \leq g^{-1} K g.
\]

Given the action of a group \( G \) on a set \( X \), we define the \( G \)-orbits and the stabilizer of elements in \( X \) as follows, for \( x \in X \):  
\[
Gx := \{ g \cdot x \mid g \in G \}, \quad G_x := \{ g \in G \mid g \cdot x = x \}.
\]
Based on the stabilizer, we then do define the following sets within \( X \). For \( H \leq G \), let:
\[
\mathcal{B}_H := \{ x \in X \mid G_x = H \},
\]
\[
\mathcal{B}_{[H]} := \{ x \in X \mid [G_x] = [H] \}.
\]
We are able to extend the partial order of conjugacy classes to these sets as:
\[
\mathcal{B}_{[H]} \leq \mathcal{B}_{[K]} \iff [H] \leq [K].
\]
Let \( X/G \) and \( \mathcal{B}_{[H]} / G \) denote the sets of orbits of the action of \( G \) on \( X \) and \( \mathcal{B}_{[H]} \), respectively. Verify that some conjugacy classes might not be included in the stabilizer set of the action of \( G \) on \( X \), so we define the set of subgroups of \( G \) that act as stabilizers for elements in \( X \) as:
\[
\StabG := \{ G_x \mid x \in X \}.
\]

It is important to note that provided that \( H \in \StabG \), then the entire conjugacy class of \( H \), \( [H] \), is contained within \( \StabG \), given that whether \( h \in G_x \) for some \( x \in X \), we acquire:
\[
(g^{-1} h g) \cdot (g^{-1} \cdot x) = g^{-1} \cdot x \quad \forall g \in G,
\]
meaning that any conjugate of \( h \) stabilizes at least one element of \( X \). We denote by \( \ConjX \) the set of conjugacy classes of subgroups of \( G \) that belong to \( \StabG \), i.e.:
\[
\ConjX := \{ [H] \in \ConjG \mid H \in \StabG \}.
\]

The following result is well-known in the theory of \( G \)-equivariant functions.

\begin{lema}\label{lema1}
Let \( G \) be a group acting on a set \( X \). Given \( x,y \in X \), the following holds:

\begin{enumerate} 
\item[i)] There exists a $ G $-equivariant function $ \tau \in \EndG $  such that $ \tau(x)=y $  if and only if $ G_{x} \leq G_{y} $.
\item[ii)] There exists a bijective $ G $-equivariant function $ \sigma \in \AutG $  such that $ \sigma(x)=y $  if and only if $ G_{x} = G_{y} $.
\end{enumerate}
\end{lema}
\noindent Further details on this result may be found in \cite{paper}.\\
An important property of the $G$-equivariant functions is given below. 

\begin{proposicion}
A constant function $f(t)=c$ is $G$-equivariant if and only if $G_{c}=G$. 
\end{proposicion}

\begin{proof}
If $f$ is a $G$-equivariant function, it holds that, given $g\in G$:
$$c=f(g\cdot t)= g\cdot f(t)=g\cdot c,\ \forall t\in X.$$
This means that $g\in G_{c}$, which implies that $G=G_{c}$.\\
On the other hand, provided that $G_{c}=G$, it arises that
$$g\cdot f(t)=g\cdot c= c= f(g\cdot t),\ \forall t\in X. $$
\end{proof}

Let us recall some important definitions of properties of functions that are useful for our objectives. Let \( f:X\rightarrow X \) be a function. The image of \( f \) is defined as  
\[
\operatorname{Im}(f):=\{f(x) \mid x\in X\}.
\]  
The kernel of \( f \) is given by:  
\[
\ker(f):=\{(x,y)\in X\times X \mid f(x)=f(y)\}.
\]  
The kernel of \( f \) defines an equivalence relation on \( X \). The following properties of the image and the kernel of a function are well known, and they shall be helpful for our goals. Given \( f,g: X\rightarrow X \), it holds that:

\begin{enumerate}
\item[$\bullet$] $Im(fg) \subseteq Im(f)$,
\item[$\bullet$] $ker(g)\subseteq ker(fg)$.
\end{enumerate}

In this document, we explore the Green's relations of the monoid of $G$-equivariant functions. Green's relations are five equivalence relations that characterize the elements of a semigroup in terms of the principal ideals they generate. Rather than working directly with a semigroup $S$, it is more convenient to define Green's relations on the monoid, which are defined as follows. Given an element in a semigroup, $a\in S$, we define the principal left, right, and bilateral (or two-sided) ideal of $a$, respectively, as follows:
$$Sa:=\{sa|\ s\in S\},$$
$$aS:=\{as|\ s\in S\},$$
$$SaS:=\{s_{1}as_{2}|\ s_{1},s_{2}\in S\}.$$
Many authors prefer to define Green's relations over a monoid instead of a semigroup, since the existence of an identity element makes some calculations easier. There exist five Green's relations, namely the $\mathcal{L}$ relation, the $\mathcal{R}$ relation, the $\mathcal{J}$ relation, the $\mathcal{H}$ relation, and the $\mathcal{D}$ relation, defined below. \\

Given two elements $a,b\in S$, we claim that:
\begin{enumerate}
\item[$\bullet$] $a \sim_{\mathcal{L}} b$ if and only if $Sa=Sb$.
\item[$\bullet$] $a \sim_{\mathcal{R}} b$ if and only if $aS=bS$.
\item[$\bullet$] $a \sim_{\mathcal{J}} b$ if and only if $SaS=SbS$.
\item[$\bullet$] $a \sim_{\mathcal{H}} b$ if and only if  $a \sim_{\mathcal{L}} b$ and $a \sim_{\mathcal{R}} b$.
\item[$\bullet$] $a \sim_{\mathcal{D}} b$ if and only if there exists an element $c\in S$ such that $a \sim_{\mathcal{L}} c$ and $c \sim_{\mathcal{R
}} b$.
\end{enumerate}

Some of the deepest properties of Green's relations can be found in \cite{Howie} and \cite{Hig}, while several fully developed examples are presented in \cite{Gould}, \cite{Gril}, and \cite{Pet}. For some propositions we shall denote $\EndG$ simply by $S$, in order to consider it the semigroup in which we are computing its Green's relations.

\begin{ejemplo}\label{ejemplo1}
Let $X=\{0,1,2,3\}$ and define an action of
$\mathbb{Z}_{2}=\{\overline 0, \overline 1\}$ over $X$ as:
$$\overline 1 \cdot x:= \left \{  \begin{array}{c} \overline 1 \cdot 0 = 0\\\overline 1 \cdot 1 = 2\\\overline 1 \cdot 2 = 1\\\overline 1 \cdot 3 = 3  \end{array} \right.$$\\

\begin{figure}[ht]\label{figurauno}
\centering
\includegraphics[width=5.5in]{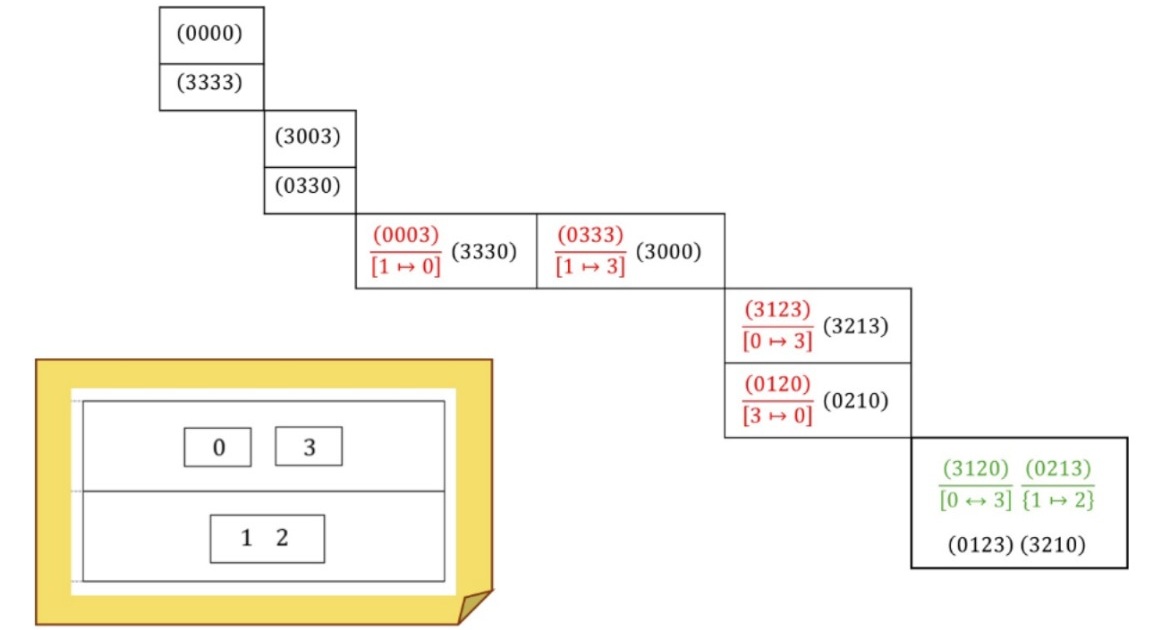}
\caption{Green's relations of a given $\mathbb{Z}_{2}$-set.}
\end{figure}

For the elements in $\EndG$ we use a transformation notations as follows:
$$\tau(t) \footnote{This is not a $G$-equivariant functions, is only illustrative for the notation.}=\left \{ \begin{array}{c} 0 \rightarrow 0 \\1 \rightarrow 3 \\2 \rightarrow 0 \\3 \rightarrow 2 \\ \end{array} \right.= (0302).$$

There exist only 16 $G$-equivariant functions in $\EndG$ for this $G$-set. After calculations, the Green's relations for this monoid generates the following structure. \\

Where the colored functions are a special kind of $G$-equivariant functions in $\EndG$, that shall be particularly addressed in the following sections. All the elements in a block are $\mathcal{D}$-related, while all the elements in the same column are $\mathcal{L}$-related, and the elements in the rows are $\mathcal{R}$-related. Consequently, the elements that are $\mathcal{H}$-related are placed in the same ''cell'', the smallest squares in the diagram.

\end{ejemplo}

There exist some interesting properties as consequences of the $G$-equivariance related to some Green's relations. We address one that shall be helpful for further results.

\begin{proposicion}\label{esta}
Let $\tau,\eta \in \EndG$ be such that $\tau \sim_{\mathcal{L}} \eta$. Then, 
$$G_{\tau(x)}=G_{\eta(x)}, \forall x\in X.$$
\end{proposicion}

\begin{proof}
Let $x\in X$ be given. Since $\tau \sim_{\mathcal{L}} \eta$, there exist $m_{1},m_{2} \in \EndG$ such that 
$$\tau=m_{1}\eta \quad \text{and} \quad \eta=m_{2}\tau.$$
As a consequence of the $G$-equivariance, it follows that
$$G_{x}\leq G_{\eta(x)} \leq G_{m_{1}\eta(x)}=G_{\tau(x)},$$
and
$$G_{x}\leq G_{\tau(x)} \leq G_{m_{2}\tau(x)}=G_{\eta(x)}.$$
Thus, we conclude that \( G_{\tau(x)}=G_{\eta(x)} \).
\end{proof}

\section{Preliminary results}
The monoid of $G$-equivariant functions is a submonoid of the full transformation monoid of the set $X$. As a submonoid, it inherits several important properties from the full transformation monoid, which are crucial for understanding its structure and behavior. In this section, we are to explore these inherited properties in detail, providing a comprehensive overview. Additionally, we shall illustrate these concepts through various examples to better demonstrate how they manifest, or do not, within the context of $G$-equivariant functions.

\begin{lema}
Let $f$ and $g$ be constant functions. Then, $f \sim_{\mathcal{L}} g$.
\end{lema}

\begin{proof}
It is sufficient to prove that $f \in Sg$, where $S$ denotes the entire monoid $\EndG$, since $f \in Sg$ implies that $Sf \subseteq Sg$, and the reverse containment follows analogously. 

The inclusion holds due to the existence of a function $\sigma$ that maps $g(t)$ to $f(t)$, which implies that  
$$ f(t) = \sigma g(t). $$
\end{proof}

\begin{lema}
There exists only one element in the $\mathcal{R}$-class of a constant functions. 
\end{lema}

\begin{proof}
Let $f$ be a constant function, and suppose that there exists $g$ such that $f\sim_{\mathcal{R}} g$. This does mean that there exists an element $m\in \EndG$ such that $g = fm$. Applying this function to an arbitrary element $t\in X$, we obtain:
\[
\begin{array}{rl} 
g(t) &= fm(t) \\ 
     &= f(t), \quad \forall t\in X, 
\end{array}
\]
hence, $f = g$.
\end{proof}

For a $G$-equivariant function $f$, its image $Im(f)$ is a $G$-invariant set, this means that $g\cdot y \in Im(f)$, for all $g\in G$ and for all $y\in Im(f)$, or equivalently $Gy\subseteq Im(f)$, for every element $y\in Im(f)$. 

\begin{lema}
For a $G$-equivariant injective function from a $G$-invariant subset of $X$ to $X$, it is possible to extend this function to a bijective $G$-equivariant function in $\EndG$. 
\end{lema}

\begin{proof}
Let be $f$ this $G$-equivariant function. It is not difficult to verify that $Im(f)$ is also a $G$-invariant subset $Y$ of $X$. If we consider a set of representatives of the orbits in each box $\Hbox$, let denote it $X_{[H]}$, we can take this set as $G_{x}=H$ for all $x\in X_{[H]}$. The restriction of $f$ to $\Hbox$ keeps being $G$-equivariant and injective functions, hence there exist bijections from $\Hbox\setminus Y$ to $\Hbox\setminus Im(f)$, moreover there exists bijections in the sets of representatives for the orbits on $\Hbox$, let call $f_{[H]}: X_{[H]} \cap (\Hbox\setminus Y) \longrightarrow X_{[H]}\cap (\Hbox\setminus Im(f))$ which extends the restriction of $f$ to the whole box $\Hbox$ as
$$\widehat{f_{[H]}}(t)=\left\{\begin{array}{cc}
f(z) & if\ z\in \Hbox \cap Y\\
 & \\
g\cdot f_{[H]}(t) &  \begin{array}{c} if\ z=g\cdot t,\\
and\ t\in X_{[H]}\cap(\Hbox \setminus Y).\end{array}  
\end{array} \right.$$
The function $\widehat{f_{[H]}}$ is a $G$-equivariant and bijective functions. A function define as
the overlaping of theses functions, $$\widehat f (t):= \widehat{f_{[H]}}(t), \mbox{ for } t\in \Hbox$$
is a bijective $G$-equivariant function which extends $f$ to $X$. 
\end{proof}

\begin{teorema}\label{L-ker}
Given two functions $f,g\in \EndG$, then $f\sim_{\mathcal{L}}g$ if and only if $ker(f)=ker(g)$. 
\end{teorema}
\begin{proof}
Whether $f\sim_{\mathcal{L}} g$, then there exist elements $m_{1},m_{2}\in \EndG$ such that 
$$f=m_{1}g,$$
$$g=m_{2}f.$$
Owing to the properties of the kernels, it follows that
$$ker(g)\subseteq ker(m_{1}g)=ker(f),$$
and $$ker(f)\subseteq ker(m_{2}f)=ker(g),$$
which means that $ker(f)=ker(g)$. \\

On the other hand, if $ker(f)=ker(g)$, this implies that $|Im(f)|=|Im(g)|$, then there exists a bijective function 
$$\overline \varphi: Im(f)\subseteq X \longrightarrow Im(g)\subseteq X,$$
given by $\overline \varphi(f(t))=g(t)$. It follows that 
$$\overline \varphi (h\cdot f(t))=\overline \varphi (f(h\cdot t))=g(h\cdot t)=h\cdot g(t)= h\cdot \overline \varphi(f(t)),$$
meaning that $\overline \varphi$ is a $G$-equivariant functions, this also implies that $\overline \varphi$ can be extended to a $G$-equivariant bijection $\varphi: X \rightarrow X$. It is easy to see that $$g=\varphi f,$$
implying that $g\in Sf$, and $Sg\subseteq Sf$. An analogous argument proves the other containment, and hence, the equality. 
\end{proof}

\begin{lema}
Given two functions $f,g\in \EndG$such that $f\sim_{\mathcal{R}}g$, then $Im(f)=Im(g)$. 
\end{lema}

\begin{proof}
If $f\sim_{\mathcal{R}} g$, then there exist elements $m_{1},m_{2}\in \EndG$ such that 
$$f=gm_{1},$$
$$g=fm_{2}.$$
Because of the properties of the Image, it follows that
$$Im(f)=Im(gm_{1})\subseteq Im(g),$$
and $$Im(g)=Im(fm_{2})\subseteq Im(f),$$
which implies that $Im(f)=Im(g)$.
\end{proof}

\begin{observacion}
In the full transformation monoid, this property is an equivalence; nevertheless, in $\EndG$, there exist elements with the same image which are not related with the $\mathcal{R}$ relation.
\end{observacion}

\begin{ejemplo}
For the same $G$-set in the example \ref{ejemplo1}, the functions in the figure acquires the same image, nonetheless, not all of them are $\mathcal{R}$-related.  
\begin{figure}[ht]
\centering
\includegraphics[width=2.5in]{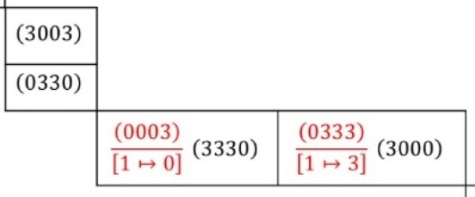}
\caption{Functions with the same image.}
\end{figure}

\end{ejemplo}

The kernel of a $G$-equivariant functions shall play a fundamental role in the results we want to accomplish, as a result, we do mention some elemental properties of them:

\begin{lema}\label{truco}
Let be $\tau\in \EndG$, then it holds that:
\begin{enumerate}
\item[i)] If $\tau(x)=\tau(y)$, then
$$\{(g\cdot x, g\cdot y),(g\cdot y, g\cdot x)|\ g\in G\}\subseteq ker(\tau)$$
\item[ii)] For every $x\in X$, it follows that$$ \{(g\cdot x, h\cdot x)|\ h^{-1}g\in G_{\tau(x)}\}\subseteq ker(\tau).$$

\end{enumerate}
\end{lema}

\begin{proof}
i) Let us suppose that $\tau(x) = \tau(y)$. Due to the $G$-equivariance, it follows that
$$g \cdot \tau(x) = g \cdot \tau(y) \implies \tau(g \cdot x) = \tau(g \cdot y) \implies (g \cdot x, g \cdot y) \in \ker(\tau).$$
The rest is due to the symmetry of the equalities. \\

ii) If $h^{-1}g \in G_{\tau(x)}$, then
$$h^{-1}g \cdot \tau(x) = \tau(x) \implies g \cdot \tau(x) = h \cdot \tau(x) \implies \tau(g \cdot x) = \tau(h \cdot x) \implies (g \cdot x, h \cdot x) \in \ker(\tau).$$
\end{proof}

The next step is to explore a special type of $G$-equivariant functions, which plays a crucial role within the monoid.

\section{Elementary collapsings}
Elementary collapsings are particular elements of the monoid of $G$-equivariant transformations, characterized by their essential role in any generating set of the monoid. The way in which these elements generate the monoid modulo its group of units can be found in \cite{paper}, but their properties have not been studied in depth in that document. We dedicate this section to listing and proving some of their most relevant properties for their application in this work.\\

Given  a subgroup $ H\leq G $  and a subset $ N\subseteq G$, we define the $ N $-conjugacy class of $ H $  as:
$$  [H]_{N}:=\{n^{-1}Hn:\ n\in N\}. $$ 
It is easy to see that $ [H]_{N} \subseteq [H] $, meaning that  the elements in an $ N $-conjugacy class of $ H $ are some of the conjugate subgroups of $ H $, specifically those given by conjugating elements in $ N $. Denote the normalizer of a subgroup $ H $  simply as $ N_{G}(H)=N_{H} $.\\

\begin{definicion}{(Elementary collapsing of type $(H,[K]_{N_{H}})$.)}\\
A $G$-equivariant function $\tau \in \EndG$ is said to be an elementary collapsing of type $(H,[K]_{N_{H}})$, whether there exist elements $x,y\in X$ satisfying the following conditions:

\begin{enumerate}
\item[i)] $Gx \neq Gy$.
\item[ii)] $G_{x}=H$.
\item[iii)] $[G_{y}]_{N_{H}}=[G_{\tau(x)}]_{N_{H}}=[K]_{N_{H}}$.
\item[iv)] 
\begin{scriptsize}
$ker(\tau)=\{(a,a)|a\in X\}\cup \{(g\cdot x,g\cdot y),(g\cdot y,g\cdot x)|g\in G\}\cup \{(g\cdot x,h\cdot x),(h\cdot x, g\cdot x)|h^{-1}g\in G_{y}\}$
\end{scriptsize}
\end{enumerate}
\end{definicion}

\begin{ejemplo}
For the $G$-set of the example \ref{ejemplo1}, the functions $\tau=(3000)$ is an elementary collapsing of type $(\langle \overline 0 \rangle, [\mathbb{Z}_{2}]_{N_{\langle 0 \rangle}})$.\\

Consider $x=1$ and $y=0$, it's easy to see that $Gx\neq Gy$. As well, is not difficult to verify that $G_{x}= \langle \overline 0 \rangle$ and $[G_{y}]_{N_{H}}=[\mathbb{Z}]_{N_{H}}$. Aside, if we compute the kernel of this functions we get
$$ker((3000))=\{(0,0),(1,1),(2,2),(3,3)\}\cup \{(1,0),(0,1),(2,0),(0,2)\}\cup\{(1,2),(2,1)\}.$$
The first part represents the reflexive part (for the equivalence relation), the secord part describes the only two different orbits that collapse to only one orbit, and the third parts characterize the elements that overlap in the orbit of $x$ when mapped to a smaller orbit.

\end{ejemplo}

The following provides an equivalent definition of elementary collapsing. However, the issue with considering elementary collapsings in this manner is that the type of collapsing is overlooked, which is an important characteristic of these functions.

\begin{proposicion}\label{resulto importante}
Given $\tau\in \EndG$, it happens that $\tau$ is an elementary collapsing of type $(G_{x},[G_{y}]_{N_{G_{x}}})$ if and only if the following statements hold:

\begin{enumerate}
\item[i)] The restriction of $\tau$ to the set $X \setminus Gx$ is an injective function. 
\item[ii)] There exists an element $z\in X$ such that $Im(\tau)=X\setminus Gz$, and $G_{x}=G_{z}$.
\end{enumerate}
\end{proposicion}

\begin{proof}
Consider $\tau$ an elementary collapsing of type $(G_{x},[G_{y}]_{N_{G_{
x}}})$. For i), let $a,b\in X\setminus Gx$ be such that $\tau (a)= \tau (b)$, this means that $(a,b)\in ker(\tau)$. As $a,b\notin Gx$, this implies that $(a,b)\notin \{(g\cdot x,g\cdot y),(g\cdot y,g\cdot x)|\ g\in G\} \cup \{(g\cdot x,h\cdot x),(h\cdot x,g\cdot x)|\ h^{-1}g\in G_{y}\}$, as $\tau$ is an elementary collapsing, then $(a,b)\in \{(a,a)|\ a\in X\}$, and thus $a=b$.\\

For ii), the restriction of $\tau$ to $\mathcal{B}_{[G_{x}]}\setminus Gx$ is an injective function, and $G$-equivariance implies that $\tau(\mathcal{B}_{[G_{x}]}\setminus Gx) \subsetneq \mathcal{B}_{[G_{x}]}$. For any two elements $t_{1},t_{2}\in \mathcal{B}_{[G_{x}]} \setminus \tau(\mathcal{B}_{[G_{x}]}\setminus Gx)$, if $Gt_{1}\neq Gt_{2}$, this would imply that there exist elements $z_{1},z_{2}\in X$ such that $\tau(Gz_{1}))=\tau(Gt_{1})$ and $\tau(Gz_{2})=\tau(Gt_{2})$, which, even if $z_{1}=z_{2}$, would mean a contradiction to the kernel of $\tau$, because it is an elementary collapsing.  Therefore, $Gt_{1}=Gt_{2}$, implying that there is only one orbit in $\mathcal{B}_{[G_{x}]}\setminus \tau(\mathcal{B}_{G_{x}}\setminus Gx)$. We consider $z\in X$ a representative of this orbit and the statement holds. As $Gz \subseteq \mathcal{B}_{G_{x}}$, it follows that $G_{x} \sim_{G} G_{z}$, we can consider $z'=g\cdot z$ to find an element in the same orbit as $z$ with the same stabilizer as $x$. \\

If $\tau$ is a function that satisfies i) and ii), $\tau(x)\in Im(\tau)$, implies that $\tau(x)\neq g\cdot z$, for all $g\in G$. The restriction of $\tau$ to $X\setminus Gx$ is a bijective function with the set $X\setminus Gz$, therefore, there exists an element $y\in X\setminus Gx$ such that $\tau(y)=\tau(x)$.  As the restriction is injective, we can assure that $G_{y}=G_{\tau(x)}$, and in consequence $[G_{y}]_{N_{G_{x}}}=[G_{\tau(x)}]_{N_{G_{x}}}$. Now, if we consider elements $g,h\in G$ such that $h^{-1}g\in G_{y}$, it follows that:
$$g\cdot y= h\cdot y \implies \tau(g\cdot y)=\tau(h\cdot y)\implies g\cdot \tau(x)=h\cdot \tau( x ) \implies (g\cdot x,h\cdot x)\in \ker(\tau).$$
Then $\tau$ is an elementary collapsing of type $(G_{x},[G_{y}]_{N_{G_{x}}})$. 
\end{proof}

Now we can characterize elemental collapsings as "almost" injective functions, which only have one orbit that overlaps with another. However, this implies the existence of a second orbit that is not part of the image. The issue with this characterization of elemental collapsings is that it does not provide information about the type of collapsing directly. For the purposes of this work, this characterization will be useful in proving certain results. However, for example, in the context addressed in \cite{paper}, it is not useful, as the collapsing type is an indispensable parameter to generate the generating sets of $\EndG$.\\

With these properties, for an elementary collapsing $\tau$ we can define a function
$$\widehat \tau(t):= \left\{\begin{array}{cc}g\cdot z & if\ t=g\cdot x \\ \tau(t) & otherwise  \end{array} \right.,$$
which is a bijective $G$-equivariant functions. We call this function the \emph{bijective support} of the elementary collapsing. We shall notice that many elementary collapsings may have the same bijective support. 

\begin{proposicion}
A constant function $f(t)=c$ is an elementary collapsing of some type if and only if $|X/G|=2$. 
\end{proposicion}

\begin{proof}
Let us suppose that $|X/G|=2$. As $f$ is a constant function, because of the $G$-equivariance, $G_{f(t)}=G$, and by the orbit-stabilizer theorem, $Gf(t)$ has only one element. therefore, there exists another element $x\in X$ such that $Gx\neq Gc$. We shall prove that 
$$ker(f)=\{(a,a)|\ a\in X\}\cup \{(g\cdot x,g\cdot y),(g\cdot y, g\cdot x)|\ g\in G\} \cup \{(g\cdot x, h\cdot x)|\ h^{-1}g\in G_{y}\}.$$ 
And $f$ shall be an elementary collapsing of type $(G_{x},[G_{c}]_{N_{G_{x}}})$, considering $y=c$ for the definition. The first containment is a consequence of the $G$-equivariance, the lemma \ref{truco}, and the fact that $f(x)=f(c)$. Now, let us consider $(a,b)\in ker(f)$ such that $a\neq b$ and $(a,b)\notin \{(g\cdot x, h\cdot x)|\ h^{-1}g\in G_{y}\}$. Note that $h^{-1}g\in G_{c}$, for all $g,h\in G$, which implies that $x$ and $c$ cannot be in the same orbit (otherwise, $e\cdot a= g\cdot b$ implies $(a,b)\in \{(g\cdot x, h\cdot x)|\ h^{-1}g\in G_{c}\}$). We may suppose, without loss of generality, that $a\in Gx$ and $b\in Gc$, i.e., $a=g\cdot x$ and $b=h\cdot c$. As $g\cdot c = h\cdot c$, for all $g,h\in G$, it follows that $(a,b)=(g\cdot x,g\cdot c)\in \{(g\cdot x,g\cdot c),(g\cdot c, g\cdot x)|\ g\in G\}$.

We shall prove the second implication by its contrapositive. Let us suppose that $|X/G|\neq2$. Provided that $|X/G|=1$, then $f$ is a bijective function. Nevertheless, as stated in the definition of elementary collapsing, these functions are required not to be injective or surjective. So let us assume that $|X/G|>2$. Given any pair $x,y\in X$, there exists at least one element $z\in X$ such that $Gx\neq Gy \neq Gz$. As $f$ is a constant function, $c=f(x)=f(z)$, which implies that $(x,z)\in ker(f)$, which is a contradiction to the fourth point of the definition of elementary collapsing.
\end{proof}

A \emph{fixing elementary collapsing} is an elementary collapsing such that additionally satisfies the following property:
$$Fix(\tau)=X\setminus Gx. $$

Based on all the information above, given $x,y\in X$ such that $G_{x}\leq G_{y}$, we define the following $G$-equivariant functions:
$$  [x\mapsto y](z)= \left\{  \begin{array}{cc}
    g\cdot y &  z=g\cdot x, \\
    z & \mbox{otherwise.}
\end{array} \right. $$ 
We may verify that if $Gx \neq Gy$, then $[x\mapsto y]$ is neither injective nor surjective. We ought to prove that it is also an elementary collapsing of type $(G_{x},[G_{y}]_{N_{G_{x}}})$.

\begin{lema}
Given $x,y\in X$ such that $Gx\neq Gy$, then $[x\mapsto y]$ is a fixing elementary collapsing of type $(G_{x},[G_{y}]_{N_{G_{x}})}$.
\end{lema}

\begin{proof}
Most of the properties of an elementary collapsing are already satisfied by the demanded conditions. The two things left to verify are the kernel of $[x\mapsto y]$ and its set of fixed points. Let $(a,b)\in \ker([x\mapsto y])$ such that $a\neq b$, and $(a,b)\notin \{(g\cdot x, g\cdot y) \mid g\in G\}$. If we suppose that $a\neq g\cdot x$ or $b\neq h\cdot x$, for all $g,h\in G$, we ought to verify the following cases:
$$
\begin{array}{c|c|c|c|c}
a=g\cdot x & a=g\cdot y & a\neq g\cdot x \neq h\cdot y & a=g\cdot x & a=h\cdot y \\
b=h\cdot y & b=h\cdot y & b\neq g\cdot x \neq h\cdot y & b\neq h\cdot x \neq k\cdot y & b\neq g\cdot x \neq k\cdot y
\end{array}
$$
All the cases lead to a contradiction, implying that $a = g \cdot x$ and $b = h \cdot x$ for some $g, h \in G$. Consequently, this implies that $h^{-1} g \in G_{y}$. \\

Given $t \in X \setminus Gx$, by definition, it follows that $[x \mapsto y](t) = t$, so $X \setminus Gx \subseteq \ker([x \mapsto y])$. If $t \notin X \setminus Gx$, then $t = g \cdot x$ for some $g \in G$. In this case, we have $[x \mapsto y](g \cdot x) = g \cdot y$, thus $t = g \cdot x \notin \text{Fix}([x \mapsto y])$, which proves the second inclusion by its contrapositive. Therefore, we conclude that $\text{Fix}([x \mapsto y]) = X \setminus Gx$, and $[x \mapsto y]$ is a fixing elementary collapsing of type $(G_{x}, [G_{y}]_{N_{G_{x}}})$.
\end{proof}

We ought to notice that if $Gx=Gy$ and $G$ is a finite group, then these functions are bijective and shall be denoted as $(x\mapsto y)$. We also define the next $G$-equivariant and bijective functions:
$$  (x\leftrightarrow y)(z)= \left\{  \begin{array}{cc}
    g\cdot y &  z=g\cdot x, \\
    g\cdot x &  z=g\cdot y, \\
    z & \mbox{otherwise.}
\end{array} \right. $$ 
which only exists if and only if $G_{x}=G_{y}$.

\begin{lema}
Every fixing elementary collapsing is of the form $[x\mapsto y]$.
\end{lema}

\begin{proof}
Let be $\tau \in\EndG$ a fixing elementary collapsing of type $(H,[K]_{N_{H}})$, there already exist elements $x,y\in X$ that satisfy the definition of an elementary collapsing. It is sufficient to prove that $\tau=[x\mapsto y]$.\\
As $Fix(\tau)=X\setminus Gx$, for all $t\notin Gx$, it happens that
$$\tau(t)=t=[x\mapsto y](t).$$
For $t\in Gx$, there exists $g\in G$ such that $t=g\cdot x$, but as $(g\cdot x, g\cdot y)\in ker(\tau)$, it holds that 
$$\tau(g\cdot x)=\tau(g\cdot y)=g\cdot y= g\cdot [x\mapsto y](x)=[x\mapsto y](g\cdot x).$$
\end{proof}

We ought to manage some properties of the elementary collapsings in order to accomplish the goals of this paper. The following are intrinsic properties of the fixing elementary collapsings.

\begin{proposicion}\label{equivalence}
It holds that $[x\mapsto y]=[g\cdot x\mapsto g\cdot y]$, for all $g\in G$.
\end{proposicion}

\begin{proof}
We must evaluate theses functions. Given $t\notin Gx$, it happens that 
$$[x\mapsto y](t)=t=[g\cdot x \mapsto g\cdot y](t).$$
If $t\in Gx$, there exists $h\in G$ such that $t=h\cdot x$, then 
$$ [x\mapsto y](h\cdot x)=h\cdot[x\mapsto y](x) =h\cdot y.$$
On the other hand, 
$$ [g\cdot x\mapsto g\cdot y](h\cdot x)=[g\cdot x\mapsto g\cdot y](h\cdot g^{-1} \cdot g\cdot x) =hg^{-1}[g\cdot x \mapsto \cdot g\cdot y](g\cdot x)$$
$$=hg^{-1}(g\cdot y)=h\cdot y.$$

\end{proof}

\begin{proposicion}
Let be $x,y,y'\in X$ such that $[x\mapsto y]$ and $[x\mapsto y']$ exist, the it holds that:
$$[x\mapsto y'][x\mapsto y]=[x\mapsto y].$$

\end{proposicion}

\begin{proof}
We ought to consider two cases, let $t\notin Gx$, note that $[x\mapsto y](t)=t,$ then 
$$t=[x\mapsto y](t)=[x\mapsto y'][x\mapsto y](t).$$
If $t\in Gx$, there exists $g\in G$ such that $t=G\cdot x$, and it happens that
$$[x\mapsto y'][x\mapsto y](g\cdot x)=  [x\mapsto y'](g\cdot y)=(g\cdot y)$$
and in the other hand
$$[x\mapsto y] (g\cdot x)=(g\cdot y)$$
which proves the equality.\end{proof}

\begin{proposicion}
Let be $x,y,\in X$ such that $[x\mapsto y]$ and $[x\mapsto y']$ exist, the it holds that:
$$[x\mapsto y][y\mapsto x]=[x\mapsto y].$$

\end{proposicion}
\begin{proof}
We ought to consider now three cases, let $t\in Gx$,  there exists $g\in G$ such that $t=G\cdot x$, then 
$$[x\mapsto y][y\mapsto x](g\cdot x)=[x\mapsto y](g\cdot x)=g\cdot y=[x\mapsto y](g\cdot x)$$
If $t\in Gy$, there exists $g\in G$ such that $t=G\cdot y$, and it happens that
$$[x\mapsto y][y\mapsto x](g\cdot y)=  [x\mapsto y](g\cdot x)=(g\cdot y)$$
and in the other hand
$$[x\mapsto y] (g\cdot y)=(g\cdot y).$$
If $t\notin Gx \cup Gy$, then 
$$[x\mapsto y][y\mapsto x](t)=[x\mapsto y](t).$$

\end{proof}

\begin{proposicion}
If $\tau$ is a fixing elementary collapsing, let say $[x\mapsto y]$, and $\eta$ an elementary collapsing of type $(G_{x},[G_{y}]_{N_{G_{x}}})$, then $\eta= \eta \tau$. 
\end{proposicion}

\begin{proof}
If $t\neq g\cdot x$, for all $g\in G$, then $\tau(t)=t$, then $\eta \tau(t)= \eta (t)$. On the other hand, if $t=g\cdot x$, for some $g\in G$, $\tau(t)=g\cdot y=g\cdot \tau(y)$. As $\eta$ is an elementary collapsings, then $\eta(x)=\eta (y)=\eta\tau(y)=\eta \tau(x)$, it follows that $\eta(t)=\eta \tau(t)$. 
\end{proof}

Now that we understand the definition of an elementary collapsing, our next goal is to describe part of the Green's relations of our monoid in terms of these highly relevant functions.

\section{Green's relations over elementary collapsings}
 These elements have proven to be of such importance that we dedicate this section of the document to describing the Green relations of the monoid based on elementary collapsings. In this section, we prove the characteristics that an element must possess in the monoid in order to be related to an elementary collapsing.

\begin{proposicion}
Let be $x,y\in X$ such that $[x\mapsto y]$ and $[y\mapsto x]$ exist, then $[x\mapsto y]\sim_{\mathcal{L}}[y\mapsto x]$.
\end{proposicion}
\begin{proof}
As $[x\mapsto y]=[x\mapsto y][y\mapsto x]$ it follows that $[x\mapsto y]\in S[y\mapsto x]$, and consequently $S[x\mapsto y]\subseteq S[y\mapsto x]$. An analogous argument proves the equality. 
\end{proof}

\begin{lema}
If there exists the elementary collapsings $[x\mapsto y]$ and $[y\mapsto x]$ then, they are not $\mathcal{R}$-related. 
\end{lema}

\begin{proof}
Let suppose that $[x\mapsto y]$ and $[y\mapsto x]$ exist, and that they are $\mathcal{R}$-related. Then there exists $m\in \EndG$ such that $[x\mapsto y]=[y\mapsto x]m$. It is easy to check that $y\notin Im([y\mapsto x])$, so $[y\mapsto x](t)\neq y$, for all $t\in X$, in particular $[y\mapsto x](m(x))\neq y$. On the other side $[x\mapsto y](x)=y$, which is a contradiction, hence $[x\mapsto y]$ is not $\mathcal{R}$-related to $[y\mapsto x]$.
\end{proof}

\begin{proposicion}
Let be $x,y,\in X$ such that $[x\mapsto y]$ and $[x\mapsto y']$ exist, then $[x\mapsto y]\sim_\mathcal{R}[x\mapsto y']$.
\end{proposicion}
\begin{proof}
As $[x\mapsto y]=[x\mapsto y'][x\mapsto y]$ it follows that $[x\mapsto y]\in [x\mapsto y']S$, and consequently $[x\mapsto y]S\subseteq [x\mapsto y']S$. Once again, an analogous argument proves the equality. 
\end{proof}

\begin{proposicion}\label{R-rel}
Given to fixing elementary collapsing of the form $[x\mapsto y]$ and $[x'\mapsto y']$ then they are related with the relation $\mathcal{R}$ if and only if there exists $g\in G$ such that $x'=g\cdot x$.
\end{proposicion}

\begin{proof}
As they are related, there exist $m_{1},m_{2}\in \EndG$ such that:
$$[x\mapsto y]=[x'\mapsto y']m_{1},$$
$$[x'\mapsto y']=[x\mapsto y]m_{2}.$$

If we suppose that $x'\neq g\cdot x$, for all $g\in G$, it follows that 
$$[x\mapsto y](x')=[x'\mapsto y']m_{1}(x')$$
$$x'=[x'\mapsto y'](m_{1}(x')).$$
If we also suppose that $m_{1}(x')=h\cdot x'$, for some $h\in G$, then it happens that 
$$x'=[x'\mapsto y'](h\cdot x')=h\cdot y'$$
which is a contradictions. Therefore $m_{1}(x')\neq h\cdot x'$, for every $g\in G$, nevertheless, this implies that
$$x'=[x'\mapsto y'](m_{1}(x'))=m_{1}(x')$$
which is also a contradiction. In consequence $x'=g\cdot x$, for some $g\in G$. \\

On the other hand, if $x'=g\cdot x$, for some $g\in G$, because of \ref{equivalence} it follows that $$[x'\mapsto y']=[g\cdot x\mapsto y']=[x\mapsto g^{-1}\cdot y],$$ 
then, as $$[x\mapsto y]=[x\mapsto g^{-1}\cdot y][x\mapsto y],$$
this implies that $[x\mapsto y]S\subseteq [x\mapsto g^{-1}\cdot y]S$. The analogous proves that second containment, and hence $[x\mapsto y]\sim_{\mathcal{R}}[x'\mapsto y']$. 
\end{proof}

\begin{teorema}
Given an elementary collapsing $\tau\in \EndG$ of type $(H,[K]_{N_{H}})$, and let $\eta\in \EndG$ be such that $\tau \sim_{\mathcal{L}} \eta$, then $\eta$ is an elementary collapsing of the same type. 
\end{teorema}

\begin{proof}
As $\tau$ is an elementary collapsing of type $(H,[K]_{N_{H}})$, there already exist elements $x,y$ such that $G_{x}=H$, $[G_{y}]_{N_{H}}=[G_{\tau(x)}]_{N_{H}}=[K]_{N_{H}}$ and , as $\tau \sim_{\mathcal{L}}\eta$ implies that $ker(\tau)=ker(\eta)$, the kernel of $\eta$ already of the form the definition demands, this means that $\eta$ is already an elementary collapsing of type $(H,[G_{\eta(x)}]_{N_{H}})$. The only missing part is to prove that $[G_{\tau(x)}]_{N_{H}}=[K]_{N_{H}}=[G_{y}]_{N_{G}}=[G_{\eta(x)}]_{N_{H}}$. As $\tau \sim_{\mathcal{L}} \eta$, because of the proposition \ref{esta}, it follows that $G_{\tau(x)}=G_{\eta(x)}$, and therefore $[G_{\tau(x)}]_{N_{H}}=[G_{\eta(x)}]_{N_{H}}$. \end{proof}

An interesting comment about this statement is that it is a necessary but not sufficient condition for being $\mathcal{L}$-related to an elementary collapsing, as it is pointed out in the following example. 

\begin{ejemplo}
For the same $G$-set in the example \ref{ejemplo1}, the functions $[1\mapsto 0]$ and $[1 \mapsto 3]$ are elementary collapsings of the same type, but they are not $\mathcal{L}$-related.  
\begin{figure}[ht]
\centering
\includegraphics[width=2.5in]{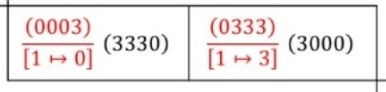}
\caption{Elementary collapsings of the same type not $\mathcal{L}$-related.}
\end{figure}

\end{ejemplo}

\begin{corolario}
Given an elementary collapsing $\tau \in \EndG$ of type $(H,[K]_{N_{H}})$, and let $\eta \in \EndG$ be such that $\tau \sim_{\mathcal{H}} \eta$, then $\eta$ is an elementary collapsing of the same type.
\end{corolario}

This follows after $\sim_{\mathcal{H}}$ implies $\sim_{\mathcal{L}}$ by its definition.

\begin{lema}\label{el chido}
Given $\tau\in \EndG$, then $\tau$ is an elementary collapsing of type $(G_{x},[G_{y}]_{N_{G_{x}}})$ if and only if there exists a fixing elementary collapsing $[x'\mapsto y']$ such that $\tau \sim_{\mathcal{R}} [x'\mapsto y']$.
\end{lema}

\begin{proof}
Let $z\in X$ the element whose orbit is not in the image of $\tau$, let consider $x'=z$ and $y'=y$, then we claim the following statements:

\begin{afirmacion}
It holds that:
$$\tau= [z\mapsto y]\tau.$$
\end{afirmacion}
As $z\notin Im(\tau)$, then $\tau(t)\neq g\cdot z$, for any $g\in G$, which implies that $[z\mapsto y]$ keeps these elements invariant, hence $[z\mapsto y]\tau(t)=\tau (t)$ for every $t\in X$ and the equality holds.  

\begin{afirmacion}
Considering the restriction of $\tau$ to $X\setminus Gx$, which is a bijection with the set $X\setminus Gz$, we can consider its inverse functions $\tau^{-1}$. Let define the function
$$m(t):=\left\{\begin{array}{cc} \tau^{-1}(t) & if\ t\notin Gz, \\ \tau^{-1}(g\cdot y)& if\ t=g\cdot z.   \end{array}\right.$$
It is easy to verify that $m$ is a $G$-equivariant function. It follows that $$[z\mapsto y]=\tau m.$$
\end{afirmacion}
Applying theses functions, for $t\notin Gz$
$$\tau m(t)=\tau \tau^{-1}(t)=t=[z\mapsto y](t).$$
For $t\in Gz$, there exists $g\in G$ such that $t=g\cdot z$, and
$$\tau m (t)=\tau \tau^{-1}(g\cdot y)=g\cdot y=[z\mapsto y](g\cdot z)=[z\mapsto y](t).$$
The first implication is a direct consequence of these two affirmations. \\

Now, let suppose that there exists a fixing elementary collapsing $[x'\mapsto y']$ that is $\mathcal{R}$-related to $\tau$. Then there exist elements $m_{1},m_{2}\in \EndG$ such that
$$\tau=[x'\mapsto y']m_{1}\mbox{ and } [x'\mapsto y']=\tau m_{2}.$$
First, note that $m_{2}$ restricted to $X \setminus Gx'$ is an injective function, if not, there exist elements $t_{1},t_{2}\in X\setminus Gx'$ (distinct) such that
$$m_{2}(t_{1})=m_{2}(t_{2}),$$
which implies that $$\tau m_{2}(t_{1})=\tau m_{2}(t_{2}) \implies [x'\mapsto y'](t_{1})=[x'\mapsto y'](t_{2}) \implies t_{1}=t_{2},$$
a contradiction to the choose of $t_{1}$ and $t_{2}$. Because of the $G$-equivariance and the injectivity on the restriction, there exists an element $z\in Z$ such that $Gz = m_{2}(X\setminus Gx')^{C}$.\\
We will prove, in the same way, that $\tau$ restricted to the set $X\setminus Gz$ is also injective. If $\tau$ is not injective in this set, there exist elements $z_{1},z_{2}\in X\setminus Gz$ such that 
$$\tau(z_{1})=\tau(z_{2}),$$
but, because of the injectivity of $m_{2}$, there exist elements $t_{1},t_{2}\in X\setminus Gx'$ satisfying that
$$z_{1}=m_{2}(t_{1}),$$
$$z_{2}=m_{2}(t_{2}).$$
It follows that 
$$\tau m_{2}(t_{1})=\tau m_{2}(t_{2})\implies [x'\mapsto y'](t_{1})=[x'\mapsto y'](t_{2}) \implies t_{1}=t_{2}.$$
Once again, a contradiction to the injectivity of $m_{2}$. \\

To finish this part of the proof we shall prove that there exists an element in $X$ such that its orbit is not in the image of $\tau$. For this purpose, because we already know that $\tau$ is injective in the set $X\setminus Gz$, it's enough to prove that $\tau(z)\neq g\cdot x'$, for every $g\in G$. Which is easy to verify after the equation $\tau=[x'\mapsto y']m_{1}$. As there exists an element $z\in X$ such that the restriction of $\tau$ to $X\setminus Gz$ is injective, and there exists an element $x'\in X$ such that its orbits is not in the image of $\tau$, by lemma \ref{resulto importante}, $\tau$ is an elementary collapsing.\\

The type of the collapsing is given by the elements $x=z$ and $y=\tau^{-1}(\tau(x))$, the preimage of $\tau(x)$ that is in the restriction to the set $X\setminus Gz$. We just have to verify the kernel of $\tau$, but, by the election of $x$ and $y$, and as consequence of the injectivity in the restriction (as in \ref{resulto importante}), it's not difficult to verify that 
\begin{small}
$$ker(\tau)=\{(a,a)| a\in X\}\cup \{(g\cdot x,g\cdot y),(g\cdot y,g\cdot x)| g\in G\} \cup \{(g\cdot x,h\cdot x),(h\cdot x,g\cdot x)|\ h^{-1}g\in G_{y}\}.$$
\end{small}
\end{proof}

\begin{corolario} 
Given elementary collapsings $\tau$ and $\eta$, $\tau \sim_{\mathcal{R}}\eta$ if and only if they miss the same orbit in their image. 
\end{corolario}

\begin{corolario}
Given an elementary collapsing $\tau\in \EndG$ of type $(H,[K]_{N_{H}})$, and let $\eta\in \EndG$ be such that $\tau \sim_{\mathcal{R}} \eta$, then $\eta$ is an elementary collapsing of some type $(H',[K']_{N_{H'}})$, moreover $H\sim_{G} H'$. 
\end{corolario}

\begin{proof}
If $\tau$ is an elementary collapsing, for \ref{el chido}, it is $\mathcal{R}$-related to a fixing elementary collapsing, afterwards, by transitivity, $\eta$ is $\mathcal{R}$-related with this fixing elementary collapsing, and, in consequence, it is also an elementary collapsing of some type.
The fact that $H\sim_{G} H'$ is a consequence of the lemma \ref{el chido} and the proposition \ref{resulto importante}. 
\end{proof}

\begin{lema}
If $\tau, \eta \in \EndG$ are elementary collapsings of type $(H,[K]_{N_{H}})$ and $(H',[K']_{N_{H'}})$ respectively, if it happens that $\tau \sim_{\mathcal{R}}\eta$ then $H \sim_{G} H'$ . 
\end{lema}

\begin{proof}
For \ref{el chido}, there exist fixing elementary collapsings $[x\mapsto y]$ and $[x'\mapsto y']$ such that $\tau \sim_{\mathcal{R}} [x\mapsto y] $ and $\eta \sim_{\mathcal{R}} [x'\mapsto y']$. By transitivity of the relations $[x\mapsto y]\sim_{\mathcal{R}}[x'\mapsto y']$, which is only posiblle if and only if $x'=g\cdot x$, for some $g\in G$, by \ref{R-rel}.
\end{proof}

An interesting observation is that being an elementary collapsing is sufficient to guarantee that the element is $\mathcal{R}$-related to a fixing elementary collapsing, and also $\mathcal{L}$-related to another fixing elementary collapsing. The situation is that these two fixing elementary collapsings may be different. In general, being an elementary collapsing is not sufficient to be $\mathcal{H}$-related to a fixing elementary collapsing.

\begin{ejemplo}
We have computed the Green's relations to a $G$-set as follows.

\begin{figure}[ht]
\centering
\includegraphics[width=5.5in]{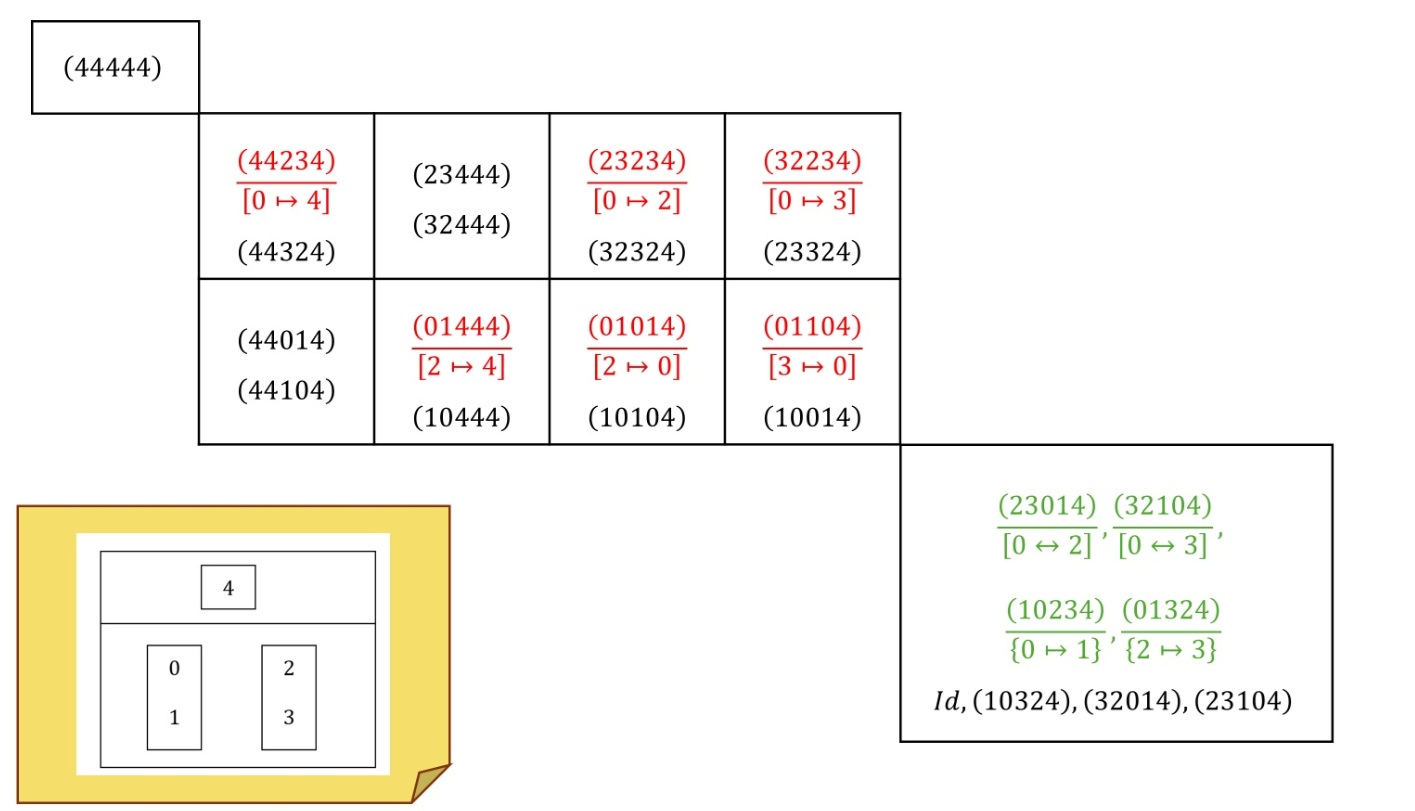}
\caption{Elementary collapsings not $\mathcal{H}$-related to a fixing elementary collapsing.}
\end{figure}

The results allow us to claim that $(23444)$, $(32444)$, $(44014)$ and $(44104)$ are elementary collapsings, but in the figure we can identify that they are not $\mathcal{H}$-related to any fixing elementary collapsing.

\end{ejemplo}

\begin{teorema}
There exist only one fixing elementary collapsing in its own $\mathcal{H}$-class.  
\end{teorema}

\begin{proof}
Let suppose that $\tau$ and $\eta$ are fixing elementary collapsings of some type, such that they are $\mathcal{H}$-related. By it's definition, there exists elements $x,x',y,y' \in X$ such that $\tau=[x\mapsto y]$
 and $\eta=[x'\mapsto y']$. As $\tau \sim_{\mathcal{R}}\eta$, because of \ref{R-rel}, there exists a $g\in G$ such that $x'=g\cdot x$. This allows us to cumpute $\eta(x)$. 
 $$y'=\eta(x')=\eta(g\cdot x)=g\cdot \eta(x) \implies \eta(x)=g^{-1}\cdot y'.$$

 Afterwards, as $\tau \sim_{\mathcal{L}}\eta$, we can assure, for \ref{L-ker}, that $ker(\tau)=ker(\eta)$ or, equivalently,  $\tau(a)=\tau(b)$ if and only if $\eta(a)=\eta(b)$. We already know that $\tau(x)=\tau(y)$, by its definition, which implies that $\eta(x)=\eta(y)$. On one side, $\eta(x)=g^{-1}\cdot y'$, and in the other hand, $\eta(y)=y$, by the fix of $\eta$, in consequence, $y'=g\cdot y$. As $x'=g\cdot x$ and $y'=g\cdot y$, the proposition \ref{equivalence} allows us to claim that $\tau=\eta$. 
 \end{proof}

\begin{corolario}
Every element in the $\mathcal{D}$-class of an elementary collapsing is also an elementary collapsing of some type.
\end{corolario}
This is a consequence of the fact that $\sim_{\mathcal{L}}$ or $\sim_{\mathcal{R}}$ implies $\sim_{\mathcal{D}}$, directly from their definitions.\\

In this work, we presented detailed characterizations and observations of the Green relations of a specific class of monoids. These relations are described in terms of particular elements within the monoid, which provides a deeper understanding of its internal structure. This approach marks a significant step forward in achieving a comprehensive understanding of the properties and structure of this mathematical object, which has thus far been relatively underexplored. Through thorough analysis, we have generalized several results, defined, and proven various particularities associated with a special type of element, the elemental collapsings, which play a key role in determining these relations. This work opens new perspectives for future research in this area and establishes a solid foundation for further studies on the theory of monoids and their Green relations.


\begin{thebibliography}{XXX}







\bibitem{cita11} Araújo, J., Mitchell, J.D.: Relative ranks in the monoid of endomorphisms of an independence algebra. Monatsh. Math. 151(1), 1–10 (2007).

\bibitem{cita21}  Bulman-Fleming, S.: Regularity and products of idempotents in endomorphism monoids of projective acts. Mathematika 42(2), 354–367 (1995).

\bibitem{cita22} Bulman-Fleming, S., Foutain, J.: Products of idempotent endomorphisms of free acts of infinite rank. Monatsh. Math. 124(1), 1–16 (1997).



\bibitem{paper} Castillo-Ramírez, A., Ruiz-Medina, R. H. (2023): The relative rank of the endomorphism monoid of a finite G-set, 
Semigroup Forum (106), Springer Science and Business Media LLC, 51-66. doi.org/10.1007


\bibitem{cita12} Dandan, Y., Dolinka, I., Gould, V.: Free idempotent generated semigroups and endomorphism monoids of free G-acts. J. Algebra 429, 133–176 (2015).



\bibitem{cita23} Fleischer, V., Knauer, U.: Endomorphism monoids of acts are wreath products of monoids with small categories. In: Jürgensen, H., Lallement, G., Weinert, H.J. (eds.) Semigroups Theory and Applications: Lecture Notes in Mathematics, vol. 1320, pp. 84–96. Springer, Berlin (1988).

\bibitem{Gould} Gould, Victoria,  Grau, Ambroise and Johnson, Marianne. (2023). The structure of End($\mathcal{T}_n$). 10.48550/arXiv.2307.11596. 

\bibitem{cita13} Gould, V.: Independence algebras. Algebra Universalis 33, 294–318 (1995).

\bibitem{cita14} Gould, V.: Independence algebras, basis algebras and semigroups of quotients. Proc. Edinburgh Math. Soc. 53(3), 697–729 (2010).
\bibitem{Gril} Grillet, Mireille P. (1970). "Green's relations in a semiring". Port. Math. 29: 181–195. Zbl 0227.16029
\bibitem{Hig} Peter M. Higgins (1992). Techniques of semigroup theory. Oxford University Press. p. 28. ISBN 978-0-19-853577-5.
\bibitem{Howie}  Howie, J. M., Schein, B. Fundamentals of Semigroup Theory.  Oxford University Press, USA, 1995.

\bibitem{cita24} Knauer, U., Mikhalev, A.V.: Endomorphism monoids of acts over monoids. Semigroup Forum 6(1), 50–58 (1973).
\bibitem{Pet} Petraq Petro (2002) Green's relations and minimal quasi-ideals in rings, Communications in Algebra 30(10): 4677–4686.

\bibitem{rotman} Rotman, Joseph J. , An Introduction to the Theory of Groups, Graduate Texts in Mathematics, Volumen 148, Springer Science $\&$ Business Media, 2012.







\end{thebibliography}
\end{document}